\documentclass{article}
\usepackage{amssymb}
\usepackage{amsfonts}
\usepackage{amsmath}
\usepackage{color}

\newtheorem{theorem}{Theorem}

\newtheorem{corollary}[theorem]{Corollary}

\newtheorem{lemma}[theorem]{Lemma}

\newenvironment{proof}[1][Proof]{\noindent\textbf{#1.} }{\ \rule{0.5em}{0.5em}}

\begin{document}

\title{Random walk with heavy tail and negative drift conditioned by its
minimum and final values}
\author{Vincent Bansaye\thanks{%
CMAP, Ecole Polytechnique, Route de Saclay, 91128 Palaiseau Cedex, France;
e-mail: bansaye@polytechnique.edu}\, and Vladimir Vatutin\thanks{%
Department of Discrete Mathematics, Steklov Mathematical Institute, 8,
Gubkin str., 119991, Moscow, Russia; e-mail: vatutin@mi.ras.ru}, }
\maketitle

\begin{abstract}
We consider random walks with finite second moment which drifts to $-\infty$
and have heavy tail. We focus on the events when the minimum and the final
value of this walk belong to some compact set. We first specify the
associated probability. Then, conditionally on such an event, we finely
describe the trajectory of the random walk. It yields a decomposition
theorem with respect to a random time giving a big jump whose distribution
can be described explicitly.
\end{abstract}

\section{Introduction and main results}

We consider a random walk $\mathbf{S}=(S_{n}:n\geq 0)$ generated by a
sequence $(X_{n}:n\geq 1)$ of i.i.d. random variables distributed as a
random variable $X$. Thus,
\begin{equation*}
S_{n}=\sum_{i=1}^{n}X_{i},\qquad (S_{0}=0).
\end{equation*}

We assume that the random walk has a negative drift
\begin{equation}
\mathbf{E}\left[ X\right] =-a<0.  \label{SubSubcritica}
\end{equation}%
and a heavy tail
\begin{equation}
A(x)=\mathbf{P}\left( X>x\right) =\frac{l(x)}{x^{\beta }},  \label{Tail0}
\end{equation}%
where $\beta >2$ and $l(x)$ is a function slowly varying at infinity. Thus,
the random variable $X$ under the measure $\mathbf{P}$ does not satisfy the
Cramer condition and has finite variance. We further suppose that, for any
fixed $\Delta >0$,
\begin{equation*}
x\left[ \frac{l(x+\Delta )}{l(x)}-1\right] \overset{x\rightarrow \infty }{%
\longrightarrow }0
\end{equation*}%
which is equivalent to
\begin{equation}
\mathbf{P}(X\in (x,x+\Delta ])=\frac{\Delta \beta \mathbf{P}(X>x)}{x}%
(1+o(1))=\frac{\Delta \beta A(x)}{x}(1+o(1))  \label{remainder}
\end{equation}%
as $x\rightarrow \infty $. \newline

To formulate the results of the present paper we introduce two important
random variables
\begin{equation*}
M_{n}\ =\ \max (S_{1},\ldots ,S_{n})\ ,\quad L_{n}\ =\ \min (S_{1},\ldots
,S_{n})
\end{equation*}%
and two right-continuous functions $U:\mathbb{R}\rightarrow \mathbb{R}%
_{0}=\left\{ x\geq 0\right\} $ and $V:\mathbb{R}\rightarrow \mathbb{R}_{0}$
given by
\begin{align*}
U(x)\ & =\ 1+\sum_{k=1}^{\infty }\mathbf{P}\left( -S_{k}\leq
x,M_{k}<0\right) \ ,\quad x\geq 0, \\
V(x)\ & =\ 1+\sum_{k=1}^{\infty }\mathbf{P}\left( -S_{k}>x,L_{k}\geq
0\right) \ ,\quad x\leq 0,\
\end{align*}%
and $0$ elsewhere. In particular $U(0)=V(0)=1$. It is well-known that $%
U(x)=O(x)$ for $x\rightarrow \infty $. Moreover, $V(-x)$ is uniformly
bounded in \thinspace $x$ in view of $\mathbf{E}X<0$.

It will be convenient to write $\mathbf{1}_{n}$ for the $n-$dimensional
vector whose coordinates are all equal to $1$ and set $\mathbf{S}%
_{j,n}=(S_{j},S_{j+1},\cdots ,S_{n})$ if $j\leq n$ with $\mathbf{S}_{n}=%
\mathbf{S}_{0,n}$ and $\mathbf{S}_{n,0}=(S_{n},S_{n-1},\cdots ,S_{0}).$
Similar notation will be used for nonrandom vectors. Say, $\mathbf{s}%
_{n,0}=(s_{n},s_{n-1},\cdots ,s_{0})$. Let%
\begin{equation*}
b_{n}=\beta \frac{\mathbf{P}\left( X>an\right) }{an}.
\end{equation*}

With this notation in hands, we first describe the asymptotic behavior of
the probability of the event that the random walk remains within the time
interval $[0,n]$ above some level $-x$ and ends up at time $n$ below the
level $T$.

\begin{theorem}
\label{C_lNonzero}For any $x\geq 0$ and $T>-x$, as $n\rightarrow \infty $,%
\begin{equation}
\mathbf{P}\left( S_{n}<T,L_{n}\geq -x\right) \sim \
b_{n}U(x)\int_{0}^{x+T}V(-z)\,dz  \label{NonZeroL}
\end{equation}%
and for any $x\geq 0$ and $T<x$, as $n\rightarrow \infty $,%
\begin{equation}
\mathbf{P}\left( S_{n}>T,M_{n}<x\right) \ \sim \
b_{n}V(-x)\int_{0}^{x-T}U(z)\,dz\ .  \label{NonzeroM}
\end{equation}
\end{theorem}

Our second goal is to demonstrate that if the event $\left\{
S_{n}<T,L_{n}\geq -x\right\} $ occurs then the trajectory of the random walk
on \thinspace $\lbrack 0,n]$ has a big jump of the order $an+O\left( \sqrt{n}%
\right) $, such a jump is unique and happens at the beginning of the
trajectory. Using this fact we also describe the full trajectory of the
random walk.

\begin{theorem}
\label{thdecomp} For all $x>0$ and $T\in \mathbb{R}$ there exists a sequence
of numbers $\pi _{j}=\pi _{j}(x)>0,\ \sum_{j\geq 0}\pi _{j}=1$, such that
for each $j$ the following properties hold:

(i) $\lim_{n\rightarrow \infty }\mathbf{P}\left( X_{j}\geq an/2|L_{n}\geq
-x,\ S_{n}\leq T\right) =\pi _{j};$

(ii) For each measurable and bounded function $F:\mathbb{R}^{j}\rightarrow
\mathbb{R}$ and each family of measurable uniformly bounded functions $F_{n}:%
\mathbb{R}^{n+1}\rightarrow \mathbb{R}$ such that
\begin{equation}
\lim_{\varepsilon \rightarrow 0}\sup_{n\in \mathbb{N},\mathbf{s}_{n}\in
\mathbb{R}^{n+1}}|F_{n}(\mathbf{s}_{n}+\epsilon \mathbf{1}_{n+1})-F_{n}(%
\mathbf{s}_{n})|=0,  \label{Ucont}
\end{equation}%
we have as $n\rightarrow \infty $
\begin{eqnarray*}
&&\mathbf{E}\left[ F(\mathbf{S}_{j-1})F_{n-j}(\mathbf{S}_{j,n})|L_{n}\geq
-x,\ S_{n}\leq T,\ X_{j}\geq an/2\right] \\
&&\quad -\mathbf{E}\left[ F(\mathbf{S}_{j-1})|L_{j-1}\geq -x\right] \mathbf{E%
}_{\mu }\left[ F_{n-j}(\mathbf{S}_{n-j,0}^{\prime })|L_{\infty }^{\prime
}\geq -x\right] \rightarrow 0,
\end{eqnarray*}%
where $\mathbf{S}^{\prime }$ is a random walk with step $-X$ and positive
drift, $L_{\infty }^{\prime }$ is its global minimum and $\mu $ is a
probability measure given by :
\begin{equation}
\mu (dy)=1_{y\in \lbrack -x,T]}\theta ^{-1}\mathbf{P}_{y}(L_{\infty
}^{\prime }\geq -x)dy,\quad \theta =\int_{-x}^{T}dy\mathbf{P}_{y}(L_{\infty
}^{\prime }\geq -x).  \label{DefMu}
\end{equation}
\end{theorem}

In words, this theorem yields the decomposition of the trajectory of $%
(S_{i}:i\leq n)$ conditioned by its minimum $L_{n}$ and final value $S_{n}$.
It says that conditionally on $L_{n}\geq -x$ and $S_{n}=s$, $\mathbf{S}$
jumps with probability $\pi _{j}$ at some (finite) time $j$. Before this
time, $\mathbf{S}$ is simply conditioned to be larger than $-x$. After this
time, reversing the trajectory yields a random walk $\mathbf{S}^{\prime }$
(with positive drift) conditioned to be larger than $-x$. The size of the
jump at time $j$ links the value $S_{j-1}$ to $S_{n-j-1}^{\prime
}=s+a(n-j-1)+\sqrt{n}W_{n}$, where, thanks to the central limit theorem, $%
W_{n}$ converges in distribution, as $n\rightarrow \infty $ to a Gaussian
random variable. Thus this big jumps is of order $an+W_{n}\sqrt{n}$, as
stated below. The proof is differed to Section \ref{preuve}. \newline

\begin{corollary}
Let $\varkappa =\inf \{j\geq 1:X_{j}\geq an/2\}$. Under $\mathbf{P}$,
conditionally on $L_{n}\geq -x$ and $S_{n}\leq T$, $\varkappa $ converges in
distribution to a proper random variable whose distribution $(\pi _{j}:j\geq
1)$ is specified by
\begin{equation*}
\pi _{j}=\pi _{j}(x)=\frac{\mathbf{P}(L_{j}\geq x)}{\sum_{k\geq 0}\mathbf{P}%
(L_{k}\geq x)}
\end{equation*}%
and
\begin{equation*}
\frac{X_{\varkappa }-an}{\sqrt{n}}
\end{equation*}%
converges in distribution to a centered Gaussian law with variance $\sigma
^{2}=Var(X)$.
\end{corollary}

\begin{proof}
The expression of $\pi _{j}$ can be found in STEP 4 of the proof of  Theorem~%
\ref{thdecomp}, see (\ref{pi}). The second part of the corollary is an
application of the second part of the mentioned theorem with
\begin{equation*}
F(s_{0},\cdots ,s_{j})=1,\qquad F_{n}(s_{1},...,s_{n+1})=g((s_{1}-an)/\sqrt{n%
})
\end{equation*}%
for $g$ uniformly continuous and bounded if one takes into account the
positivity of the drift of $\mathbf{S}^{\prime }$ allowing to neglect the
condition $L_{\infty }^{\prime }\geq -x$ and to use the central limit
theorem.
\end{proof}
$\newline$

We note that random walks with negative (or positive) drift satisfying
conditions (\ref{SubSubcritica}) and (\ref{Tail0}) (and even weaker
assumptions) have been investigated by many authors (see, for instance, 
\cite{BD94, Durr, Hir98, Kor01} and monograph \cite{BB2005}
with references therein. Article \cite{Durr} is the most close to the
subject of the present paper. Durrett has obtained there scaled limit
results for the random walk meeting conditions (\ref{SubSubcritica}) and (%
\ref{Tail0}) but conditionally on the minimum value only. He has shown that,
for each $M\geq a$ the size of the big jump may exceed the value $Mn$ with a
positive probability. The additional condition on the final value we impose
in Theorem \ref{C_lNonzero} modifies the size of the big jump by forcing it
to be concentrated in a vicinity of point $an$ with deviation of order $%
\sqrt{n}$ and allows us to provide in Theorem \ref{thdecomp} a non scaled
decomposition of the asymptotic conditional path. We stress that when the
increments of the random walk are in the domain of attraction of a Gaussian
law with zero mean, such problems have been investigated. See, in
particular, \cite{Sohier} for the convergence of the scaled random walk to
the Brownian excursion.

The initial motivations to get the results of the present paper come from
branching processes in random environment (BPRE). The survival probability
of such processes are deeply linked to the behavior of the random walk
associated with the successive $\log $ mean offspring \cite{gkv, agkv,
agkv2, abkv1, abkv2, bgk, vz}, namely $\log m$. The fine results given here
are required to get the asymptotic survival probability of the subcritical
class of BPRE such that this $\log m$ has density
\begin{equation}
p\left( x\right) =\frac{l_{0}(x)}{x^{\beta +1}}e^{-\rho x},
\label{remainder0}
\end{equation}%
where $l_{0}(x)$ is a function slowly varying at infinity, $\beta >2,$ $\rho
\in (0,1)$. We refer to \cite{VVII} for precise statements and proofs.

\section{Preliminaries : Some classical results on random walks}

\indent Our arguments essentially use a number of statements from the theory
of random walks, that are included into this section. \newline
In the sequel we shall meet the situations in which the random walk starts
from any point $x\in \mathbb{R}.$ In such cases we write for probabilities
as usual $\mathbf{P}_{x}\left( \cdot \right) .$ We use for brevity $\mathbf{P%
}$ instead of $\mathbf{P}_{0}$. \newline

We define
\begin{equation*}
\quad ~\tau _{n}=\min \left\{ 0\leq k\leq n:S_{k}=\min (0,L_{n})\right\} ,\
\tau =\min \left\{ k>0:\ S_{k}<0\right\}
\end{equation*}%
and let
\begin{equation*}
D=\sum_{k=1}^{\infty }\frac{1}{k}\mathbf{P}\left( S_{k}\geq 0\right) .
\end{equation*}%
Now we list some known statements for the convenience of references. The
first lemma is directly taken from \cite{BB2005}, Theorems 8.2.4, page~376
and 8.2.18, page~389. \newline

\begin{lemma}
\label{Ttail} Under conditions (\ref{SubSubcritica}) and (\ref{Tail0}), as $%
n\rightarrow \infty $%
\begin{equation}
\mathbf{P}\left( L_{n}\geq 0\right) =\mathbf{P}\left( \tau >n\right) \sim
e^{D}\mathbf{P}\left( X>an\right)  \label{AsL}
\end{equation}%
and for any fixed $x>0$%
\begin{equation}
\lim_{n\rightarrow \infty }\frac{\mathbf{P}\left( L_{n}\geq -x\right) }{%
\mathbf{P}\left( \tau >n\right) }=U(x).  \label{Ratio}
\end{equation}
\end{lemma}

The next statement is an easy corollary of Theorem 4.7.1, page 218 of
monograph \cite{BB2005}.

\begin{lemma}
\label{Tlocal} Let $X$ be a non-lattice random variable with $\mathbf{E}%
\left[ X\right] =-a<0$ whose distribution satisfies condition $(\ref%
{remainder})$. If $\tilde{S}_{n}=X_{1}+\cdots +X_{n}+an,$ then for any $%
\Delta >0$ uniformly in $x\geq n^{2/3}$,
\begin{equation*}
\mathbf{P}\left( \tilde{S}_{n}\in \lbrack x,x+\Delta )\right) =\frac{\Delta
\beta nA(x)}{x}(1+o(1)).
\end{equation*}
\end{lemma}

In the sequel we use several times the following lemma, in which i) does not
require a proof and ii) is a special case of Theorem 1 in \cite{cnw}.

\begin{lemma}
\label{le1} Let $(r_{n})$ be a regularly varying sequence with $%
\sum_{k=0}^{\infty }r_{k}<\infty $.

i) If $\delta _{n}\sim dr_{n}$, $\eta _{n}\sim er_{n}$, then $%
\sum_{i=0}^{n}\delta _{i}\eta _{n-i}\sim cr_{n}$ with $c=d\sum_{k=0}^{\infty
}\eta _{k}+e\sum_{k=0}^{\infty }\delta _{k}$ as $n\rightarrow \infty $.

ii) If $\sum_{k=0}^{\infty }\alpha _{k}t^{k}=\exp \big(\sum_{k=0}^{\infty
}r_{k}t^{k}\big)$ for $|t|<1$, then $\alpha _{n}\sim cr_{n}$ with $%
c=\sum_{k=0}^{\infty }\alpha _{k}$ as $n\rightarrow \infty $.
\end{lemma}

We introduce two functions
\begin{eqnarray}
K_{1}\left( \lambda \right)  &=&\frac{1}{\lambda }\exp \left\{
\sum_{n=1}^{\infty }\frac{1}{n}\mathbf{E}\left[ e^{\lambda S_{n}};S_{n}<0%
\right] \right\}   \notag \\
&=&\frac{1}{\lambda }\left( 1+\sum_{n=1}^{\infty }\mathbf{E}\left[
e^{\lambda S_{n}};M_{n}<0\right] \right) =\int_{0}^{\infty }e^{-\lambda
x}U(x)dx,  \label{DefK11}
\end{eqnarray}%
\begin{eqnarray}
K_{2}\left( \lambda \right)  &=&\frac{1}{\lambda }\exp \left\{
\sum_{n=1}^{\infty }\frac{1}{n}\mathbf{E}\left[ e^{-\lambda S_{n}};S_{n}\geq
0\right] \right\}   \notag \\
&=&\frac{1}{\lambda }\left( 1+\sum_{n=1}^{\infty }\mathbf{E}\left[
e^{-\lambda S_{n}};L_{n}\geq 0\right] \right) =\int_{0}^{\infty }e^{-\lambda
z}V(-z)dz,  \label{K22}
\end{eqnarray}%
being well defined for $\lambda >0$.

Note that the intermediate equalities in (\ref{DefK11}) and (\ref{K22}) are
simply versions of the Baxter identities (see, for instance, Chapter XVIII.3
in \cite{fe} or Chapter 8.9 in \cite{BGT}).\hfill

\section{Asymptotic behavior of the distribution of $(S_{n},L_{n})$}

\label{AP}

Basing on the three previous lemmas, we prove the following statement.

\begin{lemma}
\label{LExponent}Assume that $\mathbf{E}\left[ X\right] <0$ and that $A(x)$
meets condition~$(\ref{remainder})$. Then, for any $\lambda >0$ as $%
n\rightarrow \infty $%
\begin{equation}
\mathbf{E}\left[ e^{\lambda S_{n}};\tau _{n}=n\right] =\mathbf{E}\left[
e^{\lambda S_{n}};M_{n}<0\right] \sim K_{1}\left( \lambda \right) b_{n}
\label{AsK_1}
\end{equation}%
and%
\begin{equation}
\mathbf{E}\left[ e^{-\lambda S_{n}};\tau >n\right] =\mathbf{E}\left[
e^{-\lambda S_{n}};L_{n}\geq 0\right] \sim K_{2}\left( \lambda \right) b_{n}.
\label{AsK_2}
\end{equation}
\end{lemma}

\begin{proof}
We prove (\ref{AsK_2}) only. Statement (\ref{AsK_1}) (proved in \cite{vz}
under a bit stronger conditions) may be checked in a similar way. First we
evaluate the quantity
\begin{equation}
\mathbf{E}\left[ e^{-\lambda S_{n}};S_{n}\geq 0\right] =\mathbf{E}\left[
e^{-\lambda S_{n}};0\leq S_{n}<\lambda ^{-1}\left( \beta +2\right) \log n%
\right] +O\left( n^{-\beta -2}\right) .  \label{Fterm}
\end{equation}%
Clearly, for any $\Delta >0$%
\begin{eqnarray*}
&&\sum_{0\leq k\leq \left( \beta +2\right) \lambda ^{-1}\Delta ^{-1}\log
n}e^{-\lambda \left( k+1\right) \Delta }\mathbf{P}\left( k\Delta +an\leq
\tilde{S}_{n}\leq (k+1)\Delta +an\right) \\
&&\qquad \qquad \leq \mathbf{E}\left[ e^{-\lambda S_{n}};0\leq S_{n}<\lambda
^{-1}\left( \beta +2\right) \log n\right] \\
&&\qquad \qquad \leq \sum_{0\leq k\leq \left( \beta +2\right) \lambda
^{-1}\Delta ^{-1}\log n}e^{-\lambda k\Delta }\mathbf{P}\left( k\Delta
+an\leq \tilde{S}_{n}\leq (k+1)\Delta +an\right) .
\end{eqnarray*}%
Recall that by Lemma \ref{Tlocal} in the range of $k$ under consideration%
\begin{eqnarray*}
\mathbf{P}\left( k\Delta +an\leq \tilde{S}_{n}\leq (k+1)\Delta +an\right) &=&%
\frac{\Delta \beta n}{\left( k\Delta +an\right) }A\left( k\Delta +an\right)
(1+o(1)) \\
&=&\frac{\Delta \beta }{a}A\left( an\right) (1+o(1)),
\end{eqnarray*}%
where $o(1)$ is uniform in $0\leq k\leq \left( \beta +2\right) \lambda
^{-1}\Delta ^{-1}\log n$. Now passing to the limit as $n\rightarrow \infty $
we get%
\begin{eqnarray*}
\Delta \sum_{k=0}^{\infty }e^{-\lambda (k+1)\Delta } &\leq &\lim
\inf_{n\rightarrow \infty }\frac{a\mathbf{E}\left[ e^{-\lambda S_{n}};0\leq
S_{n}<\lambda ^{-1}\left( \beta +2\right) \log n\right] }{\beta A\left(
an\right) } \\
&\leq &\limsup_{n\rightarrow \infty }\frac{a\mathbf{E}\left[ e^{-\lambda
S_{n}};0\leq S_{n}<\lambda ^{-1}\left( \beta +2\right) \log n\right] }{\beta
A\left( an\right) } \\
&\leq &\Delta \sum_{k=0}^{\infty }e^{-\lambda k\Delta }.
\end{eqnarray*}%
Letting $\Delta \rightarrow 0+$, we see that%
\begin{equation*}
\lim_{n\rightarrow \infty }\frac{a\mathbf{E}\left[ e^{-\lambda S_{n}};0\leq
S_{n}<\lambda ^{-1}\left( \beta +2\right) \log n\right] }{\beta A\left(
an\right) }=\lambda ^{-1}.
\end{equation*}%
Combining this with (\ref{Fterm}) we conclude that, as $n\rightarrow \infty $%
\begin{equation}
\mathbf{E}\left[ e^{-\lambda S_{n}};S_{n}\geq 0\right] \sim \frac{\beta }{%
a\lambda }A\left( an\right) (1+o(1))\sim \frac{\beta }{a\lambda }\mathbf{P}%
(X>an).  \label{Nee1}
\end{equation}%
We know by the Baxter identity (see, for instance, Chapter 8.9 in \cite{BGT}%
) that for $\lambda >0$ and $t\in \lbrack 0,1]$%
\begin{equation*}
1+\sum_{n=1}^{\infty }t^{n}\mathbf{E}\left[ e^{-\lambda S_{n}};L_{n}\geq 0%
\right] =\exp \left\{ \sum_{n=1}^{\infty }\frac{t^{n}}{n}\mathbf{E}\left[
e^{-\lambda S_{n}};S_{n}\geq 0\right] \right\} .
\end{equation*}%
From (\ref{Nee1}) and point i) of Lemma \ref{le1} with $r_{n}=b_{n}$ we get
for $n\rightarrow \infty $,
\begin{equation*}
\mathbf{E}\left[ e^{-\lambda S_{n}};L_{n}\geq 0\right] \sim K_{2}\left(
\lambda \right) \frac{\beta \mathbf{P}(X>an)}{an},
\end{equation*}%
where $K_{2}\left( \lambda \right) $ is specified by (\ref{K22}). This gives
statement (\ref{AsK_2}) of the lemma.
\end{proof}

\begin{lemma}
\label{pr1} For $x\geq 0,\lambda >0$ we have as $n\rightarrow \infty :$
\begin{equation}
\mathbf{E}_{-x}[e^{\lambda S_{n}}\,;\,M_{n}<0]\ \sim \
b_{n}V(-x)\int_{0}^{\infty }e^{-\lambda z}U(z)\,dz,  \label{H1}
\end{equation}%
\begin{equation}
\mathbf{E}_{x}[e^{-\lambda S_{n}}\,;\,L_{n}\geq 0]\ \sim \
b_{n}U(x)\int_{0}^{\infty }e^{-\lambda z}V(-z)\,dz.  \label{H0}
\end{equation}
\end{lemma}

\begin{proof}
This proof follows the line for proving Proposition 2.1 in \cite{abkv1}. By
the continuity theorem for Laplace transforms Lemmas \ref{le1} and \ref%
{LExponent} give for any $x\in \lbrack 0,\infty )$ and $\lambda >0$
\begin{align}
b_{n}^{-1}\mathbf{E}[e^{\lambda S_{n}};M_{n}<0\,,\,S_{n}>-x]\ & \rightarrow
\ \int_{0}^{x}e^{-\lambda z}U(z)\,dz,  \label{Hirano4} \\
b_{n}^{-1}\mathbf{E}[e^{-\lambda S_{n}};L_{n}\geq 0\,,\,S_{n}<x]\ &
\rightarrow \ \int_{0}^{x}e^{-\lambda z}V(-z)\,dz.  \label{Hirano2}
\end{align}%
Further, using duality we have
\begin{align*}
\mathbf{E}[e^{\lambda S_{n}};M_{n}<x]\ & =\ \sum_{i=0}^{n-1}\mathbf{E}%
[e^{\lambda S_{n}};S_{0},\ldots ,S_{i}\leq S_{i}<x\,,\,S_{i}>S_{i+1},\ldots
,S_{n}] & & \\
& \quad \quad +\ \mathbf{E}[e^{\lambda S_{n}};S_{0},\ldots ,S_{n}\leq
S_{n}<x] & & \\
& =\ \sum_{i=0}^{n-1}\mathbf{E}[e^{\lambda S_{i}};L_{i}\geq 0,S_{i}<x]\cdot
\mathbf{E}[e^{\lambda S_{n-i}};M_{n-i}<0] & & \\
& \quad \quad +\mathbf{E}[e^{\lambda S_{n}};L_{n}\geq 0,S_{n}<x]. & &
\end{align*}%
This formula together with (\ref{AsK_1}), (\ref{Hirano2}), the left
continuity of $V(-z)$ for $z>0$ implying $V(0)=V(0-)=1$, and the equations
\begin{align*}
1+\sum_{k=1}^{\infty }\mathbf{E}[& e^{\lambda S_{k}};L_{k}\geq 0,S_{k}<x] \\
& =\ 1+\int_{(0,x)}e^{\lambda z}\,dV(-z)\ =\ e^{\lambda x}V(-x)-\lambda
\int_{0}^{x}e^{\lambda z}V(-z)\,dz\ , \\
1+\sum_{k=1}^{\infty }\mathbf{E}[& e^{\lambda S_{k}};M_{k}<0]\ =\ \lambda
\int_{0}^{\infty }e^{-\lambda z}U(z)\,dz
\end{align*}%
yield by Lemma \ref{le1} i) that for $\lambda >0$ and $x>0$
\begin{equation}
b_{n}^{-1}\mathbf{E}[e^{\lambda S_{n}};M_{n}<x]\ \rightarrow \
V(-x)e^{\lambda x}\int_{0}^{\infty }e^{-\lambda z}U(z)\,dz,  \notag
\end{equation}%
which gives (\ref{H1}) by multiplying by $\exp (-\lambda x)$. Using similar
arguments one can get (\ref{H0}).
\end{proof}

$\newline
$

The continuity theorem for Laplace transforms and (\ref{H1}) and (\ref{H0})
yield the asymptotic distribution of $(S_n,L_n)$ on compacts sets. \\

\begin{proof}[Proof of Theorem \protect\ref{C_lNonzero}]
By (\ref{H0}) and the continuity theorem for Laplace transforms for any $%
x\geq 0$ and $y>x$ we have%
\begin{equation*}
\mathbf{E}_{x}[e^{-\lambda S_{n}}\,;\,S_{n}<y,L_{n}\geq 0]\ \sim \
b_{n}U(x)\int_{0}^{y}e^{-\lambda z}V(-z)\,dz
\end{equation*}%
giving%
\begin{equation*}
\mathbf{P}_{x}\left( S_{n}<y,L_{n}\geq 0\right) \sim \
b_{n}U(x)\int_{0}^{y}V(-z)\,dz
\end{equation*}%
or%
\begin{equation*}
\mathbf{P}\left( S_{n}<y-x,L_{n}\geq -x\right) \sim \
b_{n}U(x)\int_{0}^{y}V(-z)\,dz
\end{equation*}%
justifying (\ref{NonZeroL}).

The asymptotic representation (\ref{NonzeroM}) may be checked by the same
arguments.
\end{proof}

\section{Conditional description of the random walk}

In this and subsequent sections we agree to denote by $C,C_{1},C_{2},...$
positive constants which may be different in different formulas or even
within one and the same complicated expression.

Our first result shows that the random walk may stay over a fixed level for
a long time only if it has at least one big jump. Let
\begin{equation*}
\mathcal{B}_{j}\left( y\right) =\{X_{j}+a\leq y\},\qquad \mathcal{B}%
^{(n)}\left( y\right) =\cap _{j=1}^{n}\mathcal{B}_{j}\left( y\right) .
\end{equation*}

\begin{lemma}
\label{L_jump} If $\mathbf{E}\left[ X\right] =-a<0$ and condition (\ref%
{remainder}) is valid then there exists $\delta _{0}\in (0,1/4)$ such that
for all $\delta \in (0,\delta _{0}),\,k\in \mathbb{Z}$, and $an/2-u\geq M,$
\begin{equation*}
\mathbf{P}_{u}(\max_{1\leq j\leq n}X_{j}\leq \delta an,\ S_{n}\geq k)\leq
\varepsilon _{M}(k)n^{-\beta -1},\text{ \ where}\qquad \varepsilon
_{M}\left( k\right) \downarrow _{M\rightarrow \infty }0.
\end{equation*}
\end{lemma}

\begin{proof}
Set $Y_{n}=\left( S_{n}+an\right) /\sigma $ where $\sigma ^{2}=Var(X)$ and $%
S_{0}=0$. It follows from Theorem 4.1.2 and Corollary 4.1.3 (i) in \cite%
{BB2005} (see also estimate (4.7.7) in the mentioned book) that if $r>2$ and
$\delta >0$ are fixed then for $x\geq n^{2/3}$ and all sufficiently large $n$
\begin{equation*}
\mathbf{P}(\mathcal{B}^{(n)}\left( x\sigma r^{-1}\right) ,Y_{n}\geq x)\leq %
\left[ n\mathbf{P}(X+a\geq \sigma xr^{-1})\right] ^{r-\delta }.
\end{equation*}%
Since $l(x)$ in (\ref{Tail0}) is slowly varying, $x^{-1/4}l(x)\rightarrow 0$
as $x\rightarrow \infty $. Hence we get for all sufficiently large $n$ and $%
\beta >2$
\begin{equation*}
\left[ n\mathbf{P}(X+a\geq \sigma xr^{-1})\right] ^{r-\delta }\leq C\left(
\frac{nl(x)}{x^{\beta }}\right) ^{r-\delta }\leq C\left( \frac{n}{x^{\beta
-1/4}}\right) ^{r-\delta }\leq C\left( \frac{1}{n^{1/6}}\right) ^{r-\delta }.
\end{equation*}%
We fix now $r>2,\delta _{0}<1/4$ with $r\delta _{0}=1/2$ so that $\left(
r-\delta \right) /6>\beta +1$ for all $\delta \in \left( 0,\delta
_{0}\right) $. As a result we obtain that there exists $\gamma >0$ such that
\begin{equation*}
\mathbf{P}_{u}(\mathcal{B}^{(n)}\left( x\sigma r^{-1}\right) ,\ S_{n}\geq
x\sigma -an+u)\leq Cn^{-\beta -1-\gamma }
\end{equation*}%
for all $x\geq n^{2/3}$ where now $S_{0}=u$. Setting $x\sigma =r\delta
_{0}an $ we get%
\begin{equation*}
\mathbf{P}_{u}(\mathcal{B}^{(n)}\left( \delta _{0}an\right) ,\ S_{n}\geq
-an/2+u)\leq Cn^{-\beta -1-\gamma }.
\end{equation*}%
Therefore, for every $k\in \mathbb{Z}$
\begin{equation}
\mathbf{P}_{u}(\max_{1\leq j\leq n}X_{j}\leq \delta _{0}an;\ S_{n}\geq
k)\leq Cn^{-\beta -1-\gamma }  \label{bigjump}
\end{equation}%
for all $an/2-u\geq M\rightarrow \infty $. Since the left-hand side is
decreasing when $\delta _{0}\downarrow 0$ the desired statement follows.
\end{proof}

$\newline
$

We know by (\ref{NonZeroL}) that for any fixed $N$ and $l\geq -N$%
\begin{equation*}
\mathbf{P}\left( L_{n}\geq -N,\,\,S_{n}\in \lbrack l,l+1)\right) \sim \
b_{n}U(N)\int_{N+l}^{N+l+1}V(-z)\,dz,\ n\rightarrow \infty .
\end{equation*}%
Hence, applying Lemma \ref{L_jump} with $u=0$ we conclude that, as $%
n\rightarrow \infty $%
\begin{equation*}
\mathbf{P}\left( L_{n}\geq -N,\,\,S_{n}\in \lbrack l,l+1)\right) \sim
\mathbf{P}\left( L_{n}\geq -N,\,\,S_{n}\in \lbrack l,l+1);\mathcal{\bar{B}}%
^{(n)}\left( \delta _{0}an\right) \right) ,
\end{equation*}%
meaning that for the event $\left\{ L_{n}\geq -N,\,\,S_{n}\in \lbrack
l,l+1)\right\} $ to occur it is necessary to have at least one jump
exceeding $\delta _{0}an$. The next statement shows that, in fact, there is
exactly one such big jump on the interval $[0,n]$ that gives the
contribution of order $b_{n}$ to (\ref{NonZeroL}) and the jump occurs at the
beginning of the interval.

\begin{lemma}
\label{L_jumpBegin} Under conditions $\mathbf{E}\left[ X\right] =-a<0$ and (%
\ref{remainder}) for any fixed $l$ and $\delta >0$
\begin{equation*}
\lim_{J\rightarrow \infty }\limsup_{n\rightarrow \infty }b_{n}^{-1}\mathbf{P}%
\left( L_{n}\geq -N,\,\max_{J\leq j\leq n}X_{j}\geq \delta an,\,S_{n}\in
\lbrack l,l+1)\right) =0
\end{equation*}%
and, for any fixed $J$
\begin{equation*}
\lim_{n\rightarrow \infty }b_{n}^{-1}\mathbf{P}\left( \cup _{i\neq
j}^{J}\left\{ X_{i}\geq \delta an,\,X_{j}\geq \delta an\right\} \right) =0.
\end{equation*}
\end{lemma}

\begin{proof}
Write for brevity $S_{n}\in \lbrack l)$ if $S_{n}\in \lbrack l,l+1)$. Then%
\begin{eqnarray*}
&&\mathbf{P}\left( L_{n}\geq -N,X_{j}\geq \delta an,S_{n}\in \lbrack
l)\right) \\
&&\quad \leq \int_{-N}^{\infty }\mathbf{P}\left( S_{j-1}\in ds,L_{j-1}\geq
-N\right) \times \\
&&\quad \qquad \quad \int_{\delta an}^{\infty }\mathbf{P}\left( X_{j}\in
dt\right) \mathbf{P}\left( S_{n-j}\in \lbrack l-t-s),L_{n-j}\geq
-t-s-N\right) \\
&&\quad \leq \int_{-N}^{\infty }\mathbf{P}\left( S_{j-1}\in ds,L_{j-1}\geq
-N\right) \int_{\delta an}^{\infty }\mathbf{P}\left( X_{j}\in dt\right)
\mathbf{P}\left( S_{n-j}\in \lbrack l-t-s)\right) .
\end{eqnarray*}%
By condition (\ref{remainder}),%
\begin{equation*}
\mathbf{P}\left( X_{j}\in \lbrack t)\right) \leq C\frac{\mathbf{P}\left(
X>t\right) }{t},\ t>0.
\end{equation*}%
This estimate and its monotonicity in $t$ gives%
\begin{eqnarray*}
&&\int_{-N}^{\infty }\mathbf{P}\left( S_{j-1}\in ds;L_{j-1}\geq -N\right)
\int_{\delta an}^{\infty }\mathbf{P}\left( X_{j}\in dt\right) \mathbf{P}%
\left( S_{n-j}\in \lbrack l-t-s)\right) \\
&\leq &C_{1}\frac{\mathbf{P}\left( X\geq \delta an\right) }{n}%
\int_{-N}^{\infty }\mathbf{P}\left( S_{j-1}\in ds;L_{j-1}\geq -N\right)
\int_{\delta an}^{\infty }\mathbf{P}\left( S_{n-j}\in \lbrack l-t-s)\right)
dt.
\end{eqnarray*}

Now%
\begin{eqnarray*}
\int_{\delta an}^{\infty }\mathbf{P}\left( S_{n-j}\in \lbrack l-t-s)\right)
dt &\leq &\int_{-\infty }^{\infty }dt\int_{l-t-s}^{l-t-s+1}\mathbf{P}\left(
S_{n-j}\in dw\right) \\
&&\quad =\int_{-\infty }^{\infty }\mathbf{P}\left( S_{n-j}\in dw\right)
\int_{l-s-w}^{l-s-w+1}dt=1.
\end{eqnarray*}%
Thus,%
\begin{eqnarray*}
&&\mathbf{P}\left( L_{n}\geq -N,X_{j}\geq \delta an,S_{n}\in \lbrack
l)\right) \\
&&\qquad \leq C\frac{\mathbf{P}\left( X>an\right) }{n}\int_{-N}^{\infty }%
\mathbf{P}\left( S_{j-1}\in ds,L_{j-1}\geq -N\right) \\
&&\qquad =C\frac{\mathbf{P}\left( X>an\right) }{n}\mathbf{P}\left(
L_{j-1}\geq -N\right) =C_{1}b_{n}\mathbf{P}\left( L_{j-1}\geq -N\right) .
\end{eqnarray*}%
By (\ref{Ratio}) the series $\sum_{j\geq 1}\mathbf{P}\left( L_{j-1}\geq
-N\right) $ converges meaning that a big jump may occur at the beginning
only. Moreover, it is unique on account of the estimate
\begin{equation*}
\mathbf{P}\left( X_{i}\geq \delta an,X_{j}\geq \delta an\right) =O\left(
l^{2}(n)n^{-2\beta }\right) =o\left( b_{n}\right)
\end{equation*}%
for all $i\neq j$ with $\max \left( i,j\right) \leq J$ and $\beta >2$.
\end{proof}

\bigskip The next lemma gives an additional information about the properties
of the random walk in the presence of a big jump. Let%
\begin{equation*}
\mathcal{R}_{\delta }\left( M,K\right) =\left\{ \delta an\leq X_{1}\leq an-M%
\sqrt{n},\,\left\vert S_{n}\right\vert \leq K\right\}
\end{equation*}%
and
\begin{equation*}
\mathcal{R}\left( M,K\right) =\left\{ X_{1}\geq an+M\sqrt{n},\,\left\vert
S_{n}\right\vert \leq K\right\} .
\end{equation*}

\begin{lemma}
\label{L_replace}Under conditions $\mathbf{E}\left[ X\right] =-a<0$ and (\ref%
{remainder}) for any $\delta \in (0,1)$ and each fixed $K$,%
\begin{equation*}
\lim_{M\rightarrow \infty }\limsup_{n\rightarrow \infty }b_{n}^{-1}\mathbf{P}%
\left( \mathcal{R}_{\delta }\left( M,K\right) \cup \mathcal{R}\left(
M,K\right) \right) =0.
\end{equation*}
\end{lemma}

\begin{proof}
Similarly to the previous lemma we have
\begin{eqnarray*}
\mathbf{P}\left( \mathcal{R}_{\delta }\left( M,K\right) \right)
&=&\int_{\delta an}^{an-M\sqrt{n}}\mathbf{P}\left( S_{n-1}\in \left[ -K-x,K-x%
\right] \right) \mathbf{P}\left( X_{1}\in dx\right) \\
&\leq &C\frac{\mathbf{P}\left( X>\delta an\right) }{\delta an}\int_{\delta
an}^{an-M\sqrt{n}}\mathbf{P}\left( S_{n-1}\in \left[ -2K-x,2K-x\right]
\right) dx \\
&=&C\frac{\mathbf{P}\left( X>\delta an\right) }{\delta an}\int_{\delta
an}^{an-M\sqrt{n}}dx\int_{-2K-x}^{2K-x}\mathbf{P}\left( S_{n-1}\in dv\right)
\\
&\leq &4KC\frac{\mathbf{P}\left( X>\delta an\right) }{\delta an}%
\int_{-2K-an+M\sqrt{n}}^{2K-\delta an}\mathbf{P}\left( S_{n-1}\in dv\right)
\\
&\leq &4KC\frac{\mathbf{P}\left( X>\delta an\right) }{\delta an}\mathbf{P}%
\left( S_{n-1}\geq -2K-an+M\sqrt{n}\right)
\end{eqnarray*}%
and%
\begin{eqnarray*}
\mathbf{P}\left( \mathcal{R}\left( M,K\right) \right) &=&\int_{an+M\sqrt{n}%
}^{\infty }\mathbf{P}\left( S_{n-1}\in \left[ -K-x,K-x\right] \right)
\mathbf{P}\left( X_{1}\in dx\right) \\
&\leq &C\frac{\mathbf{P}\left( X>an\right) }{an}\int_{an+M\sqrt{n}}^{\infty }%
\mathbf{P}\left( S_{n-1}\in \left[ -2K-x,2K-x\right] \right) dx \\
&\leq &4KC\frac{\mathbf{P}\left( X>an\right) }{an}\mathbf{P}\left(
S_{n-1}\leq 2K-an-M\sqrt{n}\right) .
\end{eqnarray*}%
Since $\limsup_{n\rightarrow \infty }\mathbf{P}\left( \left\vert
S_{n-1}+an\right\vert \geq M\sqrt{n}\right) $ decreases to $0$ as $%
M\rightarrow \infty $ by the central limit theorem, the desired statement
follows.
\end{proof}

\section{Proof of Theorem \protect\ref{thdecomp}\label{preuve}}

We start by the following important statement.

\begin{lemma}
\label{unif} Let $F_{n}$ be a bounded family of uniformly equicontinuous
functions as defined in Theorem \ref{thdecomp} by (\ref{Ucont}). Then the
family of functions
\begin{equation*}
g_{n}(s)=\sqrt{n}\mathbf{E}_{s}\left[ F_{n}(\mathbf{S}_{n});L_{n}\geq
-x,S_{n}\leq T\right] ,\quad n=1,2,...,
\end{equation*}
is uniformly equicontinuous and uniformly bounded in $s\in \mathbb{R}$.
\end{lemma}

\begin{proof}
First, the fact that the family of functions $F_{n}$ is bounded by $C$
combined with the Stone local limit theorem for iid random variables having
finite variance (see, for instance, \cite{BGT}, Section 8.4) allows us to
bound $g_{n}$ by
\begin{equation*}
C\sqrt{n}\mathbf{P}_{s}(S_{n}\in \lbrack -x,T])=C\sqrt{n}\mathbf{P}(S_{n}\in
\lbrack -x-s,T-s])\leq C_{1}<\infty .
\end{equation*}%
Second,
\begin{eqnarray*}
&&|g_{n}(s+\epsilon )-g_{n}(s)| \\
&&\quad =\sqrt{n}|\mathbf{E}_{s}\left[ F_{n}(\mathbf{S}_{n}+\epsilon \mathbf{%
1}_{n+1});L_{n}+\epsilon \geq -x,S_{n}+\epsilon \leq T\right] \\
&&\quad \quad \qquad -\mathbf{E}_{s}\left[ F_{n}(\mathbf{S}_{n});L_{n}\geq
-x,S_{n}\leq T\right] | \\
&&\quad \ \leq \sqrt{n}|\mathbf{E}_{s}\left[ F_{n}(\mathbf{S}_{n}+\epsilon
\mathbf{1}_{n+1})-F_{n}(\mathbf{S}_{n});L_{n}+\epsilon \geq
-x,S_{n}+\epsilon \leq T\right] | \\
&&\quad \qquad \quad \quad \quad +\sqrt{n}|\mathbf{E}_{s}\left[ F_{n}(%
\mathbf{S}_{n});L_{n}+\epsilon \geq -x,S_{n}+\epsilon \leq T\right] \\
&&\qquad \qquad \qquad \quad \qquad \qquad -\mathbf{E}_{s}\left[ F_{n}(%
\mathbf{S}_{n});L_{n}\geq -x,S_{n}\leq T\right] | \\
&&\quad \leq H_{\epsilon }C+\sqrt{n}|\mathbf{P}_{s}(L_{n}+\epsilon \geq
-x,S_{n}+\epsilon \leq T)-\mathbf{P}_{s}(L_{n}\geq -x,S_{n}\leq T)|,
\end{eqnarray*}%
where $H_{\epsilon }\rightarrow 0$ as $\epsilon \rightarrow 0$ again by the
assumptions on $F_{n}$ and the Stone local limit theorem. Let us prove now
that the last term is small. Indeed,
\begin{eqnarray*}
&&\sqrt{n}|\mathbf{P}_{s}(L_{n}+\epsilon \geq -x,S_{n}+\epsilon \leq T)-%
\mathbf{P}_{s}(L_{n}\geq -x,S_{n}\leq T)| \\
&&\qquad \leq \sqrt{n}\left[ \mathbf{P}_{s}(S_{n}\in \lbrack T-\epsilon ,T])+%
\mathbf{P}_{s}(L_{n}\in \lbrack -x-\epsilon ,-x[,S_{n}\leq T)\right]
\end{eqnarray*}%
and only the second term raises a difficulty. By the total probability
formula with respect to the (first) time $k$ of the minimum \ we have
\begin{eqnarray*}
&&\mathbf{P}_{s}(L_{n}\in \lbrack -x-\epsilon ,-x[,S_{n}\leq T) \\
&&\quad =\mathbf{P}(L_{n}+s+x\in \lbrack -\epsilon ,0[,S_{n}\leq T-s) \\
&&\quad =\sum_{k=0}^{n}\mathbf{P}(S_{1}>S_{k},\cdots
,S_{k-1}>S_{k},S_{k}+s+x\in \lbrack -\epsilon ,0), \\
&&\qquad \qquad \qquad \qquad \qquad \qquad S_{k+1}\geq S_{k},\cdots
,S_{n}\geq S_{k},S_{n}\leq T-s) \\
&&\quad \leq \sum_{k=0}^{n-\left[ \sqrt{n}\right] -1}\mathbf{P}(S_{k}+s+x\in
\lbrack -\epsilon ,0))\mathbf{P}(S_{k+1}\geq S_{k},\cdots ,S_{n}\geq S_{k})
\\
&&\qquad +\sum_{k=n-\left[ \sqrt{n}\right] }^{n}\mathbf{P}%
(S_{1}>S_{k},\cdots ,S_{k-1}>S_{k},S_{k}+s+x\in \lbrack -\epsilon ,0)) \\
&&\qquad \qquad \qquad \qquad \qquad \qquad \qquad \qquad \times \mathbf{P}%
(S_{k+1}\geq S_{k},\cdots ,S_{n}\geq S_{k}).
\end{eqnarray*}%
Now we use the representation
\begin{equation*}
\mathbf{P}(S_{k+1}\geq S_{k},\cdots ,S_{n}\geq S_{k})=\mathbf{P}(L_{n-k}\geq
0)\sim C(n-k+1)^{-\beta }
\end{equation*}%
and the Stone local limit theorem according to which
\begin{equation*}
\sqrt{2\pi n}\mathbf{P}(S_{k}+s+x\in \lbrack -\epsilon ,0[)=\epsilon \exp
\left\{ -\frac{\left( s+x\right) ^{2}}{2\sigma ^{2}n}\right\} +\delta _{n},
\end{equation*}%
where $\delta _{n}\rightarrow 0$ as $n\rightarrow \infty $ uniformly in $%
k\in \lbrack n-\sqrt{n},n]$ and $s+x\in \mathbb{R}$. Hence we conclude that
\begin{eqnarray*}
&&\sqrt{n}\mathbf{P}_{s}(L_{n}\in \lbrack -x-\epsilon ,-x[,S_{n}\leq T) \\
&&\qquad \leq \left( \epsilon +\delta _{n}\right) C\sum_{k=0}^{n-\left[
\sqrt{n}\right] -1}(n-k+1)^{-\beta }+C_{1}\sum_{k=n-\left[ \sqrt{n}\right]
}^{n}(n-k+1)^{-\beta } \\
&&\qquad \leq C_{2}\left( \epsilon +\delta _{n}+\sqrt{n}(\sqrt{n})^{-\beta
}\right) \leq C_{3}\left( \epsilon +\delta _{n}\right) ,
\end{eqnarray*}%
for $n$ large enough, since $\beta >1$. We end up the proof by noting that
all these bounds are uniform with respect to $s$.
\end{proof}

$\newline
$

\begin{proof}[Proof of Theorem \protect\ref{thdecomp}]
We know by Lemmas \ref{L_jump}, \ref{L_jumpBegin} and \ref{L_replace} that
conditionally on the event $\{L_{n}\geq -x,S_{n}\leq T\}$, there is a
(single) big jump, that its size is of order $an$ with a deviation of order$%
\sqrt{n}$ and that the jump happens at the beginning. Taking this into
account and setting $\mathcal{A}_{j}^{M}=\{X_{j}-an\in \lbrack -M\sqrt{n},M%
\sqrt{n}]\}$ we get
\begin{equation*}
\mathbf{E}\left[ F(\mathbf{S}_{j-1})F_{n-j}(\mathbf{S}_{j,n});L_{n}\geq
-x,S_{n}\leq T\right] =\varepsilon _{J,M,n}b_{n}+\sum_{j=0}^{J}A_{j,n}^{M},
\end{equation*}%
where
\begin{equation*}
\lim_{J,M\rightarrow \infty }\sup_{n}\left\vert \varepsilon
_{J,M,n}\right\vert =0
\end{equation*}%
and
\begin{equation*}
A_{j,n}^{M}=\mathbf{E}\left[ F(\mathbf{S}_{j-1})F_{n-j}(\mathbf{S}%
_{j,n});L_{n}\geq -x,\mathcal{A}_{j}^{M},S_{n}\leq T\right] .
\end{equation*}%
By the Markov property we have
\begin{equation}
A_{j,n}^{M}=\mathbf{E}\left[ F(\mathbf{S}_{j-1})1_{\left\{ L_{j-1}\geq
-x\right\} }H_{j,n}^{M}(S_{j-1})\right] ,  \notag
\end{equation}%
where%
\begin{equation*}
H_{j,n}^{M}(s)=\mathbf{E}\left[ 1_{\mathcal{A}^{M}}\mathbf{E}_{s+X}\left[
F_{n-j}(\mathbf{S}_{n-j});L_{n-j}\geq -x,S_{n-j}\leq T\right] \right]
\end{equation*}%
and $\mathcal{A}^{M}=\{X-an\in \lbrack -M\sqrt{n},M\sqrt{n}]\}.$

STEP 1. We are proceeding by bounded convergence and show first the simple
convergence. Thus, we consider
\begin{eqnarray}
b_{n}^{-1}H_{j,n}^{M}(s) &=&\int_{-M\sqrt{n}+an}^{M\sqrt{n}+an}b_{n}^{-1}%
\mathbf{P}(X\in dy)\mathbf{E}_{s+y}\left[ F_{n-j}(\mathbf{S}%
_{n-j});L_{n-j}\geq -x,S_{n-j}\leq T\right]  \notag \\
&=&\frac{1}{\sqrt{n}}\int_{-M\sqrt{n}+an}^{M\sqrt{n}+an}g_{j,n}(s+y)\mu
_{n}(dy),  \notag
\end{eqnarray}%
where
\begin{equation*}
\mu _{n}(dy)=b_{n}^{-1}\mathbf{P}(X\in dy),\quad g_{j,n}(s)=\sqrt{n}\mathbf{E%
}_{s}\left[ F_{n-j}(\mathbf{S}_{n-j});L_{n-j}\geq -x,S_{n-j}\leq T\right] .
\end{equation*}%
We want to prove that
\begin{equation*}
\frac{1}{\sqrt{n}}\int_{-M\sqrt{n}+an}^{M\sqrt{n}+an}g_{j,n}(s+y)\mu
_{n}(dy)-\frac{1}{\sqrt{n}}\int_{-M\sqrt{n}+an}^{M\sqrt{n}%
+an}g_{j,n}(s+y)dy\rightarrow 0
\end{equation*}%
as $n\rightarrow \infty $ by using the local converges of $\mu _{n}$ to the
Lebesgue measure (with uniformity in $y\in \left[ an-M\sqrt{n},an+M\sqrt{n}%
\right] $ thanks to (\ref{Tail0}) and (\ref{remainder})) and the uniform
equicontinuity of $g_{j,n}$ (compare with Lemma \ref{unif}). Let us give the
details. \newline
First, by Lemma \ref{unif} for any $\varepsilon >0$ there exists $\eta >0$
such that for all $n\geq n_{0}=n_{0}(\varepsilon ,\eta )$ we have
\begin{equation*}
\sup_{y}\sup_{u\in \lbrack 0,\eta ]}|g_{j,n}(y)-g_{j,n}(y+u)|\leq \epsilon .
\end{equation*}%
Let, further, $s_{i}(=s_{i}^{n})$ be a division of $[an-M\sqrt{n},an+M\sqrt{n%
}-1]$ into subintervals with step $\eta $. Then, for sufficiently large $%
n\geq n_{0}$,
\begin{equation*}
\left\vert b_{n}^{-1}H_{j,n}^{M}(s)-\frac{1}{\sqrt{n}}\sum_{i}g_{j,n}(s_{i})%
\mu _{n}[s_{i},s_{i+1})\right\vert \leq 3M\epsilon .
\end{equation*}%
Besides, $g_{j,n}(y)$ is bounded by $C$ with respect to the pair $n,y$ by
Lemma \ref{unif}. Recalling that by (\ref{remainder})
\begin{equation*}
\sup_{y\in s+an+[-M\sqrt{n},M\sqrt{n}]}|\mu _{n}[y,y+\eta ]-\eta |\leq
\epsilon \eta
\end{equation*}%
for $n$ large enough, we get
\begin{equation*}
\left\vert b_{n}^{-1}H_{j,n}^{M}(s)-\frac{1}{\sqrt{n}}\sum_{i}g_{j,n}(s_{i})%
\eta \right\vert \leq 3M\epsilon +2CM\epsilon \eta \frac{1}{\eta }.
\end{equation*}%
Using again the uniform continuity of $g_{j,n}$ yields for $n$ large enough
\begin{equation*}
\left\vert \frac{1}{\sqrt{n}}\sum_{i}g_{j,n}(s_{i})\eta -\frac{1}{\sqrt{n}}%
\int_{-M\sqrt{n}+an}^{M\sqrt{n}+an}g_{j,n}(s+y)dy\right\vert \leq 3\epsilon
M,
\end{equation*}%
resulting in
\begin{equation*}
\left\vert b_{n}^{-1}H_{j,n}^{M}(s)-\frac{1}{\sqrt{n}}\int_{-M\sqrt{n}+an}^{M%
\sqrt{n}+an}g_{j,n}(s+y)dy\right\vert \leq M(6+2C)\epsilon
\end{equation*}%
for $n$ large enough.

Clearly, $b_{n}^{-1}H_{j,n}^{M},n=j+1.j+2,...$ is a bounded sequence since
both $g_{j,n}$ (see Lemma \ref{unif}) and $\mu _{n}([an-M\sqrt{n},an+M\sqrt{n%
}])/\sqrt{n}$ are bounded. This and the dominated convergence theorem lead
to
\begin{equation}
b_{n}^{-1}A_{j,n}^{M}-\mathbf{E}\left[ F(\mathbf{S}_{j-1})1_{\left\{
L_{j-1}\geq -x\right\} }\frac{1}{\sqrt{n}}\int_{-M\sqrt{n}+an}^{M\sqrt{n}%
+an}g_{j,n}(S_{j-1}+y)dy\right] \overset{n\rightarrow \infty }{%
\longrightarrow }0.  \label{asA}
\end{equation}

STEP 2. We can now complete the proof by reversing the random walk after
time $j$. To this aim set $\mathbf{s}_{k}=\left( s_{0},\cdots ,s_{k}\right) $
and $\mathbf{s}_{n,0}=\left( s_{n},\cdots ,s_{0}\right) $ and recall that
(see, for instance, Lemma 9 in \cite{Hir98})
\begin{equation*}
ds_{0}\mathbf{P}_{s_{0}}(\mathbf{S}_{n}\in d\mathbf{s}_{n})=ds_{n}\mathbf{P}%
_{s_{n}}(\mathbf{S}_{n,0}^{\prime }\in d\mathbf{s}_{n,0}).
\end{equation*}%
Hence, letting
\begin{equation*}
\mathcal{B}_{n}(\mathbf{s}_{k})=\{\left\vert s_{0}-an-s\right\vert \leq M%
\sqrt{n},\,s_{k}\in \lbrack -x,T],\min_{0\leq i\leq k}s_{i}\geq -x\}
\end{equation*}%
we get by integration
\begin{eqnarray*}
&&\int 1_{\mathcal{B}_{n}(\mathbf{s}_{n-j})}F_{n-j}(\mathbf{s}_{n-j})ds_{0}%
\mathbf{P}_{s_{0}}(\mathbf{S}_{n-j}\in d\mathbf{s}_{n-j}) \\
&&\quad =\int 1_{\mathcal{B}_{n}(\mathbf{s}_{n-j})}F_{n-j}(\mathbf{s}%
_{n-j})ds_{n-j}\mathbf{P}_{s_{n-j}}(\mathbf{S}_{n-j,0}^{\prime }\in d\mathbf{%
s}_{n-j,0}).
\end{eqnarray*}%
It follows that
\begin{eqnarray*}
&&\frac{1}{\sqrt{n}}\int_{-M\sqrt{n}+an}^{M\sqrt{n}+an}g_{j,n}(s+y)dy \\
&&\ \ =\int_{s_{0}^{\prime }\in \lbrack -x,T]}\mathbf{E}_{s_{0}^{\prime }}%
\left[ F_{n-j}(\mathbf{S}_{n-j,0}^{\prime });L_{n-j}^{\prime }\geq
-x;\left\vert S_{n-j}^{\prime }-an-s\right\vert \leq M\sqrt{n}\right]
ds_{0}^{\prime }.
\end{eqnarray*}%
Since, as $n\rightarrow \infty $
\begin{equation*}
\mathbf{P}(\left\vert S_{n-j}^{\prime }-an-s\right\vert \leq M\sqrt{n}%
)\rightarrow \frac{1}{\sigma \sqrt{2\pi }}\int_{-M}^{M}\exp \left\{ -\frac{%
y^{2}}{2\sigma ^{2}}\right\} dy
\end{equation*}%
for every $s\in \mathbb{R}$ and $\mathbf{S}^{\prime }$ has a positive drift,
we conclude that
\begin{eqnarray*}
&&K^{M}(s)=\limsup_{n\rightarrow \infty }\bigg\vert\frac{1}{\sqrt{n}}\int_{-M%
\sqrt{n}+an}^{M\sqrt{n}+an}g_{j,n}(s+y)dy \\
&&\qquad \qquad \qquad \qquad -\int_{s_{0}^{\prime }\in \lbrack -x,T]}%
\mathbf{E}_{s_{0}^{\prime }}\left[ F_{n-j}(\mathbf{S}_{n-j,0}^{\prime
});L_{\infty }^{\prime }\geq -x\right] ds_{0}^{\prime }\bigg\vert
\end{eqnarray*}%
goes to $0$ as $M$ becomes large. Further, by the bounded convergence and
taking into account the boundness of $g_{j,n}$, we get (recall (\ref{DefMu}%
))
\begin{eqnarray*}
&&\limsup_{n\rightarrow \infty }\bigg\vert\mathbf{E}\left[ F(\mathbf{S}%
_{j-1})1_{\left\{ L_{j-1}\geq -x\right\} }\frac{1}{\sqrt{n}}\int_{-M\sqrt{n}%
+an}^{M\sqrt{n}+an}g_{j,n}(S_{j-1}+y)dy\right]  \\
&&\qquad \qquad -\mathbf{E}\left[ F(\mathbf{S}_{j-1})1_{\left\{ L_{j-1}\geq
-x\right\} }\theta \mathbf{E}_{\mu }\left[ F_{n-j}(\mathbf{S}%
_{n-j,0}^{\prime })|L_{\infty }^{\prime }\geq -x\right] \ \right] \bigg\vert
\\
&&\quad \qquad \qquad \qquad \qquad \qquad \qquad \leq \mathbf{E}\left[ F(%
\mathbf{S}_{j-1})1_{\left\{ L_{j-1}\geq -x\right\} }K^{M}(S_{j-1})\right]
\end{eqnarray*}%
where the right-hand side goes to $0$ as $M\rightarrow \infty $. Using (\ref%
{DefMu}) once again we set
\begin{equation*}
D_{j,n}=\mathbf{E}\left[ F(\mathbf{S}_{j-1})|L_{j-1}\geq -x\right] \mathbf{E}%
_{\mu }\left[ F_{n-j}(\mathbf{S}_{n-j}^{\prime })|L_{\infty }^{\prime }\geq
-x\right]
\end{equation*}%
and deduce from $(\ref{asA})$ that the function
\begin{equation*}
R^{M}=\limsup_{n\rightarrow \infty }\left\vert b_{n}^{-1}A_{j,n}^{M}-\mathbf{%
P}(L_{j-1}\geq -x)\theta D_{j,n}\right\vert
\end{equation*}%
goes to zero as $M\rightarrow \infty $. Writing
\begin{equation*}
C_{j,n}=\mathbf{E}\left[ F(\mathbf{S}_{j-1})F_{n-j}(\mathbf{S}%
_{j,n});L_{n}\geq -x;\ S_{n}\leq T;\ X_{j}\geq an/2\right]
\end{equation*}%
we have
\begin{eqnarray*}
&&\limsup_{n\rightarrow \infty }\left\vert b_{n}^{-1}C_{j,n}-\mathbf{P}%
(L_{j-1}\geq -x)\theta D_{j,n}\right\vert  \\
&&\qquad \leq \limsup_{n\rightarrow \infty }b_{n}^{-1}\mathbf{P}(X_{j}\geq
an/2,|X_{j}-an|>M\sqrt{n},\ S_{n}\leq T)+R^{M}.
\end{eqnarray*}%
Combining the last limit and Lemma \ref{L_replace} ensures that the
right-hand side of this inequality goes to $0$ as $M\rightarrow \infty $. We
conclude
\begin{equation*}
\lim_{n\rightarrow \infty }\left( b_{n}^{-1}C_{j,n}-\mathbf{P}(L_{j-1}\geq
-x)\theta D_{j,n}\right) =0.
\end{equation*}

STEP 4. We apply the limit above to the family of functions $F=1,F_{n-j}=1$
and get
\begin{equation}
b_{n}^{-1}\mathbf{P}\left( L_{n}\geq -x,\ S_{n}\leq T,\ X_{j}\geq
an/2\right) \overset{n\rightarrow \infty }{\longrightarrow }\mathbf{P}%
(L_{j-1}\geq -x)\theta .  \label{pi}
\end{equation}%
Recalling (\ref{NonZeroL}) ensures that there exists $\pi _{j}(x)>0$ such
that
\begin{equation*}
\mathbf{P}\left( X_{j}\geq an/2\ |\ L_{n}\geq -x,\ S_{n}\leq T\right)
\overset{n\rightarrow \infty }{\longrightarrow }\pi _{j}(x).
\end{equation*}%
Using Lemmas \ref{L_jump}, \ref{L_jumpBegin} and \ref{L_replace} shows that
there is only one big jump at the beginning, and it has to be greater than $%
an/2$. Thus, $\sum_{j\geq 0}\pi _{j}(x)=1$. Finally, the proof of the
Theorem can be completed by using again the conclusion of STEP 3. \newline

\textbf{Acknowledgement.} This work was partially funded by the project
MANEGE `Mod\`{e}les Al\'{e}atoires en \'{E}cologie, G\'{e}n\'{e}tique et
\'{E}volution' 09-BLAN-0215 of ANR (French national research agency), Chair
Modelisation Mathematique et Biodiversite VEOLIA-Ecole
Polytechnique-MNHN-F.X. and the professorial chair Jean Marjoulet. The
second author was also supported by the Program of the Russian Academy of
Sciences \textquotedblleft Dynamical systems and control
theory\textquotedblright .
\end{proof}

\end{document}